\documentclass[12pt]{amsart}
\usepackage[all]{xy}
\usepackage{amsfonts}
\usepackage{amssymb}

\usepackage{color}

\newtheorem{thm}{Theorem}
\newtheorem{lem}[thm]{Lemma}
\newtheorem{prop}[thm]{Proposition}

\newtheorem{cor}[thm]{Corollary}

\newtheorem{dfn}[thm]{Definition}
\newtheorem{rmk}[thm]{Remark}

\def\<{\langle}
\def\>{\rangle}

\def\id{\mathop{\rm id }\nolimits}

\begin{document}

\title{On extendability of functionals on Hilbert $C^*$-modules}
\author{V. Manuilov}

\date{}

\address{Moscow Center for Fundamental and Applied Mathematics, Moscow State University,
Leninskie Gory 1, Moscow, 119991, Russia\\ \newline \indent Harbin Institute of Technology, Xida St. 92, Harbin, Heilongjiang, 150001, China}

\email{manuilov@mech.math.msu.su}

\thanks{The research was supported by RSF, project No. 21-11-00080}

\subjclass[2010]{46L08}

\maketitle

\begin{abstract}
Let $M\subset N$ be Hilbert $C^*$-modules over a $C^*$-algebra $A$ with $M^\perp=0$. It
was shown recently by J. Kaad and M. Skeide that there exists a non-zero $A$-valued functional on $N$
such that its restriction onto $M$ is zero. Here we show that this may happen even if $A$ is
monotone complete. On the other hand, we show that for certain type I $W^*$-algebras this
cannot happen.

\end{abstract}

\section{Introduction}

Hilbert $C^*$-modules were introduced by W. Paschke in \cite{Pas1} and became a useful tool in
various aspects of noncommutative geometry and topology. The idea behind this notion is
to generalize the notion of a Hilbert space by replacing the field of the complex numbers,
both as scalars and as the range for the inner product, by a $C^*$-algebra $A$. For details about
Hilbert $C^*$-modules we refer to \cite{Lance,MTBook,MT}. In some aspects, Hilbert $C^*$-modules behave
like Hilbert spaces, but there are three major differences.

First, while any element $m\in M$ of a Hilbert $C^*$-module $M$ determines an $A$-antilinear
map $x\mapsto \langle x,m\rangle = \widehat m(x)$, $\widehat m : M \to A$, not any bounded $A$-antilinear map from $M$ to $A$ is
of this form, i.e. there is no Riesz theorem. The $A$-module of all bounded $A$-antilinear maps
is called the dual module and is denoted by $M'$. The map $m\mapsto \widehat m$ determines an isometric
inclusion $M\subset M'$, which is often not surjective. As a corollary, bounded $A$-linear operators
on Hilbert $C^*$-modules need not have an adjoint. This explains the interest in special classes
of Hilbert $C^*$-modules like self-dual or reflexive ones.

The dual module $M'$ often has no Hilbert $C^*$-module structure. Nevertheless, when the
underlying $C^*$-algebra $A$ is a $W^*$-algebra, or, slightly more generally, a monotone complete
$C^*$-algebra, the inner product $\langle\cdot,\cdot\rangle$ extends from $M$ to $M'$ in a natural way, \cite{F_1995}.

Second, while there is only one (up to an isometric isomorphism) separable Hilbert space
$l_2(\mathbb C)$, a countably generated Hilbert $C^*$-module over $A$ need not be isometrically isomorphic
to the standard Hilbert $C^*$-module $l_2(A)$.

Third, a Hilbert $C^*$-submodule $M\subset N$ in a Hilbert $C^*$-module $N$ need not be (orthogonally)
complementable. It may easily happen that the orthogonal complement $M^\perp$ is zero,
while $M$ is strictly smaller than $N$. We call such submodules {\it thick}.

Any element $n\subset N$ determines a functional $\widehat n|_M$ on $M$ given by $\widehat n|_M(m)=\langle m, n\rangle$. A
long-standing problem was to check whether the norms $\|n\|$ and $\|\widehat n|_M\|$ coincide. If $A$ is a
$W^*$-algebra and if $N=M'$ (in this case $M'$ is a Hilbert $C^*$-module) then it was shown in
\cite{Pas1} that these two norms are the same.

Recently, a surprising counterexample appeared in \cite{KS}, showing that for $A=C[0; 1]$ and
for $M=l_2(A)$, there exists $N$ such that $M$ is thick in $N$, and the two norms differ. As
a concequence, it was shown in \cite{KS} that there exist non-zero functionals on $N$ with zero
restriction onto $M$, thus showing that thick submodules may be not so thick.

Hilbert $C^*$-modules over $C[0,1]$ are reflexive, i.e. the second dual module $M''$ coincides with $M$, but the first dual module is not a Hilbert one, which is a reason behind a lot of counterexamples. In this paper, we focus on the class of Hilbert $C^*$-modules, for which $M'=M''$ is a Hilbert $C^*$-module. In particular, this is the case when $A$ is a monotone complete $C^*$-algebra \cite{F_1995}.

We show that even when $A$ is monotone complete, the two norms may be different. On the other hand, we show that if $A=L^\infty(X)$ is a commutative $W^*$-algebra or $A=\mathbb B(H)$ is the algebra of all bounded operators on a Hilbert
space then the two norms coincide. We have no idea on what happens when $A$ is a general
$W^*$-algebra.

In order to compare the norms, we associate to each element $f$ of the dual module, a right
ideal $J_f$ in $A$ and show that the size of this ideal is related to the coincidence of the two
norms.

The author is grateful to M. Frank and E. Troitsky for fruitful discussions.

\section{Hilbert $C^*$-modules with Hilbert dual}

Recall that the map $m\mapsto \widehat m$, $M\to M'$, defines an isometric inclusion, so we may identify $M$ with a subset in $M'$ via this map.

\begin{dfn}
A Hilbert $C^*$-module $M$ over a $C^*$-algebra $A$ is called a \textit{module with a Hilbert dual} if the inner product $\langle\cdot,\cdot\rangle$ on $M$ extends to an inner product $\langle\cdot,\cdot\rangle'$ on $M'$ such that $\langle \widehat m,\widehat n\rangle'=\langle m,n\rangle$ and $f(m)=\langle \widehat m,f\rangle'$ for any $m,n\in M$ and any $f\in M'$.

\end{dfn}

\begin{lem}
Let $M$ be a Hilbert $C^*$-module with a Hilbert dual. Then $M'$ is self-dual.

\end{lem}
\begin{proof}
The identity map extends to the inclusion $M''\subset M'$ for any Hilbert $C^*$-module. If $M'$ is a Hilbert $C^*$-module then the identity map extends to the inclusion $M'\subset M''$. Hence $M''=M'$.

\end{proof}

\begin{prop}
Let $\langle\cdot,\cdot\rangle'$ and $\langle\cdot,\cdot\rangle''$ be two such extensions of the inner product on $M$ to inner products on $M'$. Then they coincide.

\end{prop}
\begin{proof}

Set $\langle f,g\rangle^+=\frac{1}{2}(\langle f,g\rangle'+\langle f,g\rangle'')$, $f,g\in M'$. Then $\langle\cdot,\cdot\rangle^+$ is also an extension of the inner product on $M$. Then $\langle f,f\rangle'\leq 2\langle f,f\rangle^+$ for any $f\in M'$, hence $\id:M'\to M'$ is a bounded map from $(M',\langle,\cdot,\cdot\rangle^+)$ to $(M',\langle,\cdot,\cdot\rangle')$. It is also obviously $A$-linear.
 
As $M'$ with any inner product is self-dual, there exists the adjoint operator $\id^*:(M',\langle,\cdot,\cdot\rangle')\to (M',\langle,\cdot,\cdot\rangle^+)$. By definition, $\langle \id f,g\rangle'=\langle f,\id^*g\rangle^+$ for any $f,g\in M'$. In particular, this holds for $f=\widehat m$, $m\in M$. Then
$$
\langle \id \widehat m,g\rangle^+=\langle \widehat m,g\rangle=\langle \widehat m,\id^* g\rangle=\langle \widehat m,\id^* g \rangle'
$$
for any $g\in M'$ and any $m\in M$, hence $\id^*g=\id g$, i.e. $\id^*=\id$, therefore $\langle f,g\rangle'=\langle \id f,\id g\rangle^+$ for any $f,g\in M'$. Thus, $\langle f,g\rangle'=\langle f,g\rangle''$ for any $f,g\in M'$.

\end{proof}

Any self-dual Hilbert $C^*$-module and any Hilbert $C^*$-module over a monotone complete $C^*$-algebra is obviously a module with a Hilbert dual \cite{F_1995}.

\section{Right ideals from functionals}

Let $f\in M'$. Define the subset $J_f\subset A$ by $J_f=\{a\in A:fa\in M\}$. 
\begin{lem}
$J_f$ is a closed right ideal in $A$.

\end{lem}
\begin{proof} 
Obvious.
\end{proof}

Apart from the trivial case when $f\in M$ (then $M_f=A$), there are two extreme cases for the right ideal $J_f$ when $f\notin M$:
\begin{itemize}
\item
$J_f$ is zero;
\item
$J_f$ is essential, i.e. $aJ_f=0$, $a\in A$, implies $a=0$.
\end{itemize}

Let $M$ be a Hilbert $C^*$-module over a $C^*$-algebra $A$, $f\in M'$. Define $M^0_f$ as the $A$-linear span of $M$ and $f$, i.e. 
$$
M^0_f=\{m+fa:m\in M,a\in A\}\subset M', 
$$
and let $M_f$ be the norm closure of $M^0_f$. Usually, $M_f$ need not be a pre-Hilbert $C^*$-module, i.e. the inner product on $M$ need not be extendable to $M_f$, but sometimes this happens, e.g. when $A$ is a monotone complete $C^*$-algebra \cite{F_1995}. We are interested in the uniqueness of the inner product extension from $M$ to $M_f$, or to larger modules, e.g. to $M'$. 


\begin{lem}
Suppose that the inner product on $M$ extends to an inner product on $M_f$. If the right ideal $J_f$ is essential then the extended inner product on $M_f$ is unique.

\end{lem}
\begin{proof}
Let $\langle\cdot,\cdot\rangle_1$ and $\langle\cdot,\cdot\rangle_2$ be two extensions of the inner product $\langle\cdot,\cdot\rangle$ on $M$ making $M_f$ a pre-Hilbert $C^*$-module. Set $a_i=\langle f,f\rangle_i$, $i=1,2$. If $j\in J_f$ then $fj\in M$, hence $f(fj)=\langle fj,f\rangle_i$, $i=1,2$, hence $j^*a_1=j^*a_2$, or, equivalently, passing to the adjoints, $(a_2-a_1)j=0$. As this holds for any $j\in J_f$ and as $J_f$ is essential then $a_2-a_1=0$, hence the two inner products coincide.

\end{proof}


\begin{thm}\label{=}
Suppose that $M$ is a Hilbert $C^*$-module with a Hilbert dual, and let $\langle\cdot,\cdot\rangle_{M'}$ be the extention of the inner product $\langle\cdot,\cdot\rangle$ from $M$ to $M'$. Suppose also that $M$ is a thick submodule in a Hilbert $C^*$-module $N$ with the inner product $\langle\cdot,\cdot\rangle_N$, which determines an injection $i:N\to M'$ (the inner product on $M$ is supposed to be the restriction of the inner product on $N$). If the right ideal $J_f$ is essential for any $f\in M'$ then the inner product on $N$ coincides with the extended inner product on $M'$, i.e. $\langle n,n'\rangle_N=\langle i(n),i(n')\rangle_{M'}$ for any $n,n'\in N$.

\end{thm}
\begin{proof}
Recall that the injection $j$ is defined as follows: for $n\in N$, set $i(n)(m)=\langle m,n\rangle_N$, $m\in M$, i.e. $i(n)=\widehat n|_M$. Thickness of $M$ in $N$ implies that $i$ is injective.

Suppose that there exists $n\in N$ such that $\langle n,n\rangle_N\neq\langle i(n),i(n)\rangle_{M'}$. Set $f=i(n)\in M'$, and let $j\in J_f$. Then, by definition, $m=i(n)j\in M$, and
$$
\langle m',nj\rangle_N=\langle m',nj\rangle_N=i(nj)(m')=\langle m',i(n)j\rangle_{M'},\quad m'\in M
$$
hence $\langle m',nj-m\rangle_N=0$ for any $m'\in M$. As $M$ is thick in $N$, the latter implies that $nj=m\in M$. 

Set $a=\langle n,n\rangle_N$, $b=\langle i(n),i(n)\rangle_{M'}$. Then 
$$
j^*aj=\langle nj,nj\rangle_N=\langle m,m\rangle_M=\langle i(n)j,i(n)j\rangle_{M'}=j^*bj,
$$ 
hence $j^*(a-b)j=0$ for any $j\in J_f$. As $J_f$ is essential, this implies that $a=b$. Thus, $\langle n,n\rangle_N=\langle i(n),i(n)\rangle_{M'}$ for any $n\in N$. Polarization identity then implies that $\langle n',n\rangle_N=\langle i(n'),i(n)\rangle_{M'}$ for any $n,n'\in N$.

\end{proof}

\begin{thm}\label{neq}
Suppose that $m$ is a Hilbert $C^*$-module with a Hilbert dual, and let $\langle\cdot,\cdot\rangle'$ be the extension of the inner product $\langle\cdot,\cdot\rangle$ from $M$ to $M'$, hence to $M_f\subset M'$. 
Suppose also that $\langle f,f\rangle'\geq 1$.
If $J_f=0$ then there exists a Hilbert $C^*$-module $N$ such that $M\subset N$ is thick, $f\in N$, and the norm of $f$ as a functional (i.e. the standard norm on $M'$) is strictly smaller than the norm of $f$ as an element of $N$.

\end{thm}
\begin{proof}
 The condition $J_f=0$ means that the subspaces $M$ and $fA=\{fa:a\in A\}$ form a direct sum (note that it need not be a topological direct sum) equal to $M^0_f$. Let $c=\langle f,f\rangle'\in A$, and, for $m+fa,m'+fa'\in M^0_f$, set 
$$
\langle m+fa,m'+fa'\rangle_+=\langle m+fa,m'+fa'\rangle'+a^*c a'.
$$ 

Let us check linearity with respect to the first argument: 
\begin{eqnarray*}
\langle m_1{+}m_2{+}f(a_1{+}a_2),m'{+}fa'\rangle_{+}&\!=\!&\langle m_1{+}m_2{+}f(a_1{+}a_2),m'{+}fa'\rangle'{+}(a_1{+}a_2)ca'\\
&\!=\!&
\langle m_1{+}fa_1,m'{+}fa'\rangle'{+}\langle m_2{+}fa_2,m'{+}fa'\rangle'{+}(a_1{+}a_2)ca'\\
&\!=\!&
\langle m_1{+}fa_1,m'{+}fa'\rangle_{+}{+}\langle m_2{+}fa_2,m'{+}fa'\rangle_{+}.
\end{eqnarray*}
Other properties of an inner product obviously hold, making thus $M^0_f$ a pre-Hilbert $C^*$-module with respect to $\langle\cdot,\cdot\rangle_+$.

Let us show that $M^0_f$ is complete with respect to the norm $\|\cdot\|_+$ determined by the inner product $\langle\cdot,\cdot\rangle_+$ (it may be not complete with respect to the norm on $M'$).

Suppose that $\{m_k+fa_k\}_{k\in\mathbb N}$ is a Cauchy sequence with respect to $\|\cdot\|_+$, where $m_k\in M$, $a_k\in A$, $k\in\mathbb N$. Note that $\|fa\|_+\leq\|m+fa\|_+$ for any $m\in M$, $a\in A$, therefore, $\{fa_k\}_{k\in\mathbb N}$ is a Cauchy sequence as well. Then $\{m_k\}_{k\in\mathbb N}$ is Cauchy too.

As $\langle f,f\rangle'\geq 1$, we have
$$
a^*a\leq a^*\langle f,f\rangle'a=\langle fa,fa\rangle'
$$
for any $a\in A$, hence
$$
\|a_k-a_n\|^2\leq \langle fa_k-fa_n,fa_k-fa_n\rangle',
$$
so $\{ak\}_{k\in\mathbb N}$ is a Cauchy sequence too. As both $A$ and $M$ are complete, the limit of the
sequence $\{m_k + fa_k\}_{k\in\mathbb N}$ is of the form $m + fa$ for some $m\in M$ and $a\in A$, thus $M^0_
f$ is complete.

Note that 
$$
\langle m,m'+fa\rangle_+=\langle m,m'+fa\rangle'=\langle m,m'\rangle+f(m)a,\quad m,m'\in M, a\in A,
$$
hence $\langle\cdot,\cdot\rangle_+$ is an extension of the inner product $\langle,\cdot,\cdot\rangle$ on $M$ to $N$.
As $M$ is thick in $M'$, it is thick in $N$ as well. Indeed, if $\langle m,m'+fa\rangle_+=0$ for any $m\in M$ then
$\langle m,m'+fa\rangle'=0$, hence $m'+fa\in M'$ is orthogonal to the whole $M$, hence $m'+fa=0$.

Finally, 
$$
\|f\|'=(\langle f,f\rangle')^{1/2}=\|c\|^{1/2}, 
$$
while 
$$
\|f\|_+=(\langle f,f\rangle_+)^{1/2}=(\langle f,f\rangle'+c)^{1/2}=\sqrt{2}\|c\|^{1/2}.
$$ 

\end{proof}

\begin{rmk}
Considering $\widetilde M=A\oplus M$ instead of $M$, and $\widetilde f = (\widehat 1_A, f)\in \widetilde M'$ instead of
$f\in M'$, we get $\langle \widetilde f, \widetilde f\rangle'\geq 1$, and $J_{\widetilde f}=J_f$, so the condition $\langle f,f\rangle'\geq 1$ is not restrictive.

\end{rmk}

\section{Positive results}

For a Hilbert $C^*$-module $N$ we write $\widehat n$ for the functional $x\mapsto\langle n,x\rangle$.

\begin{thm}
Let $M\subset N$ be Hilbert $C^*$-modules over a commutative $W^*$-algebra $A=C(X)$ such that $M$ is countably generated and $M^\perp=0$. Then $\|n\|=\|\widehat n|_M\|$ for any $n\in N$,i.e. the map $n\mapsto\widehat n|_M$, $N\to M'$, is an isometry. 

\end{thm}
\begin{proof}
If $M$ is countably generated then it can be considered as a direct summand in the standard Hilbert $C^*$-module $l_2(A)$, i.e. $l_2(A)=M\oplus K$ for some Hilbert $C^*$-module $K$. The functional $\widehat n|_M$, $n\in N$, can be similarly considered as a functional on $l_2(A)$ (defined by zero on $K$). If $\|n\|\neq\|\widehat n|_M\|$ for some $n\in N$ then the two norms differ for $(n,0)\in N\oplus K$, so the proof reduces to the case when $M=l_2(A)$. 

Let $f\in M'$, and let $J_f=\{a\in A: fa\in M\}$ be an ideal in $A$. We claim that if $M=l_2(A)$ then $J_f$ is essential for any $f\in M'$, i.e. $aJ_f=0$, $a\in A$, implies that $a=0$. Let $f=(f_n)_{n\in\mathbb N}$, $f_n\in C(X)$. The condition $f\in M'$ means that $\sum_{n=1}^\infty |f_n|^2$ is bounded. Set $g_n=\sum_{k=1}^n|f_n|^2$. As $A$ is a $W^*$-algebra, the least upper bound $\bar g$ of the monotone sequence $(g_n)_{n\in\mathbb N}$ exists and coincides with the pointwise limit 
$\check g=\lim_{n\to\infty}g_n$
outside a meager subset $Y_f$ of $X$. We claim that $J_f$ consists of all functions $a$ vanishing on $Y_f$. Recall that $fa\in M$ iff the sequence $\{f_na\}_{n\in\mathbb N}$ is norm convergent.

Supppose that $a\in A$, $a|_{Y_f}=0$. Given $\varepsilon>0$, set $U=\{x\in X:|a(x)|<\frac{\varepsilon}{\|\bar g\|}\}$. 
Clearly, $U\subset X$ is open, and $Y_f\subset U$. Then $|g_n(x)a(x)-\bar{g}(x)a(x)|\leq\varepsilon$ for any $x\in U$ and for
any $n\in\mathbb N$. On the compact set $X\setminus U$, the sequence $\{g_na\}_{n\in\mathbb N}$ converges pointwise to the
continuous function $\bar g a$, hence the convergence is uniform --- both on $U$ and on $X\setminus U$. Thus,
$a\in J_f$.

Now, let $a\in J_f$. Suppose that there exists $x_0\in Y_f$ such that $a(x_0)\neq 0$. The sequence
$\{g_na\}_{n\in\mathbb N}$ is uniformly convergent on $X$, hence $\check ga$ is continuous and coincides with $\bar ga$, i.e.
$\bar g(x)a(x)=\check g(x)a(x)$ for any $x\in X$. But $\bar g(x_0)\neq \check g(x_0)$. This contradiction shows that
$a|_{Y_f}=0$.

Let us show that the ideal $J_f$ is essential. Let $aJ_f=0$ for some $a\in C(X)$. Let $V=\{x\in X:a(x)\neq 0\}$. Then $V$ is an open subset of $X$. As $ab=0$ for any $b\in J_f$, we should have $b(x)=0$ for any $x\in V$. Recall that if $C(X)$ is a $W^*$-algebra then $X$ is a hyperstonean space. By Corollary following Proposition 5 of \cite{Dixmier}, in hyperstonean spaces, any meager set is nowhere dense, i.e. the interior of its closure is empty. Recall that functions in $J_f$ equal zero on $Y_f$, hence, on its closure $\overline{Y}_f$. The condition $aJ_f=0$ means that for any $b\in J_f$ one has
$b|_V=0$. We claim that this implies $V\subset \overline{Y}_f$. Indeed, assume the contrary: let $x_0\in V\setminus\overline{Y}_f$ satisfy the condition: any continuous function that vanishes on $\overline{Y}_f$ vanishes at $x_0$. As $X$ is
compact and Hausdorff, it is normal, hence there exists a continuous function that vanishes
on $\overline{Y}_f$ and equals 1 at $x_0$. This contradiction proves $V\subset\overline{Y}_f$. But $Y_f$ is nowhere dense,
hence the open set $V$ must be empty. Hence $a=0$.

\end{proof}


Let $J$ be of $c_0$-$\sum_{i}\oplus \mathbb K(H_i)$-type, i.e. it has a faithful $*$-representation as a $C^*$-algebra of compact operators on some Hilbert space, and let $A\supset J$ be a $C^*$-algebra such that $J$ is an essential ideal in $A$. As examples, we may consider $A=\mathbb B(H)$, $J=\mathbb K(H)$ and $A=l_\infty$, $J=c_0$.

\begin{thm}
Let $M\subset N$ be Hilbert $C^*$-modules over $A$ such that $M^\perp=0$. Suppose that $M$ is a Hilbert $C^*$-module with a Hilbert dual. If $A$ is a $C^*$-algebra containing an essential ideal $J$ of $c_0$-$\sum_{i}\oplus \mathbb K(H_i)$-type then $\|n\|=\|\widehat n|_M\|$ for any $n\in N$, i.e. the map $n\mapsto\widehat n|_M$, $N\to M'$, is an isometry. 

\end{thm}
\begin{proof}
Denote by $MJ\subset M$ the norm closure of the linear span of $\{mj:m\in M,j\in J\}$. Then $MJ$ is a Hilbert $C^*$-module over $J$. Clearly, $MJ\subset NJ$. It was shown in \cite{Magajna} that any submodule over $J$ is orthogonally complemented.

We claim that $MJ=NJ$. Indeed, if this is not true then  the orthogonal complement for $MJ$ in $NJ$ is not zero. Let $x\in NJ$ satisfy $x\perp MJ$. Then $\langle mj,x\rangle=0$ for any $m\in M$ and any $j\in J$. As $J$ is essential, we can conclude that $\langle m,x\rangle=0$ for any $m\in M$. As $M^\perp=0$ in $N$, we conclude that $x=0$.
It follows that $nj\in M$ for any $n\in N$ and any $j\in J$.  

Similarly, consider the inclusion $M\subset M'$. Here we have $M^\perp=0$ as well. Passing to $MJ\subset M'J$, we can conclude that $fj\in M$ for any $f\in M'$. In particular, $\widehat n|_M\cdot j\in M$.

By assumption, $M'$ is a Hilbert $C^*$-module with the inner product $\langle\cdot,\cdot\rangle_{M'}$. By Theorem \ref{=}, 
$\langle \widehat n|_M,\widehat n|_M\rangle_{M'}=\langle n,n\rangle_N$ for any $n\in N$.

\end{proof}

\section{Negative results}

Here we show that the norm of extensions of functionals from thick submodules may be not unique even for commutative monotone complete $C^*$-algebras. Take $X=[0,1]$, and consider the $C^*$-algebra $Bor(X)$ of all bounded Borel functions with the sup-norm. Consider the set $J$ of all functions $a\in Bor(X)$ such that the set $\{x\in X:a(x)\neq 0\}$ is meager. The set $J$ is an ideal in $Bor(X)$, so let $A=B(X)=Bor(X)/J$ be the quotient $C^*$-algebra. We
denote the equivalence class of $a$ by $[a]\in A$.
The $C^*$-algebra $B(X)$ was showed to be monotone complete in \cite{Dixmier}, while $Bor(X)$ is only monotone $\sigma$-complete \cite{S-W}.

\begin{lem}\label{Borel}
There exists $f\in l_2(A)'$ with $\langle f,f\rangle'\geq 1$ such that $fa\in l_2(A)$ for $a\in A$ implies that $a=0$. 

\end{lem}
\begin{proof}
We use the idea from \cite{KS} to construct $f=([f_1],[f_2],\ldots)\in l_2(A)'$ in such a way that the least upper bound for the partial sums of $\sum_{k=1}^\infty |f_k|^2$ would differ from their pointwise limit on the set of rationals in $(0,1)$. Here $f_k\in Bor(X)$, $k\in\mathbb N$.

Let $(r_m)_{m\in\mathbb N}$ be the sequence of all rational numbers in $(0,1)$ enumerated by natural numbers. For each $r_m$, take a strictly increasing sequence $\alpha_m^k\in(0,1)$ such that $\lim_{k\to\infty}\alpha_m^k=r_m$, and let $h_m^k$ be the characteristic function of the interval $[\alpha_m^k,\alpha_m^{k+1})$. Fix a bijection 
$$
\varphi:\mathbb N\to\mathbb N^2,\quad \varphi(n)=(m(n),k(n)), 
$$
and set $f_n=\frac{1}{2^{m(n)}}h_{m(n)}^{k(n)}$. Then the partial sums $\sum_{n=1}^N |f_n|^2$ (and hence $\sum_{n=1}^N [|f_n|]^2$) are uniformly bounded, hence $f=([f_n])_{n\in\mathbb N}\in l_2(A)'$. Note that we could take continuous functions instead of $h_n^k$ to get a required $f$. To fit the condition $\langle f,f\rangle'\geq 1$ we may replace the constructed $f=([f_1],[f_2],\ldots)$ by $(1_A,[f_1],[f_2],\ldots)$. For shortness' sake set $g_m=\sum_{k=1}^m|f_k|^2\in Bor(X)$. 

The series $\sum_{k=1}^n[|f_k|]^2$ is not norm convergent in $B(X)$ (i.e. $f\notin l_2(A)$). Indeed, note first that $\|\chi\|_{Bor(X)}=\|[\chi]\|_{B(X)}=1$ for a characteristic function $\chi$ of any open subset of $X$. The norm in the quotient algebra is given by 
$$
\|[\chi]\|_{B(X)}=\inf_{\psi\in J}\|\chi+\psi\|_{Bor(X)}=1
$$ 
as $\psi$ vanishes outside of a meager set. Take any $N\in\mathbb N$. There exists $n_0\in\mathbb N$ such that $m(N+n_0)=1$. Then, for any $M>N+n_0$, we have $g_M-g_N\geq\frac{1}{2}h_1^{k(N+n_0)}$, hence 
$$
\|[g_M]-[g_N]\|\geq\frac{1}{2}\|[h_1^{k(N+n_0)}]\|=\frac{1}{2}.
$$

Moreover, this series is not norm convergent on any closed interval $[c,d]=V\subset X$. Let $\check q:Bor(X)\to Bor(V)$ denote the restriction map. It induces the map of the quotients $q:B(X)\to B(V)$. We claim that the series $\sum_{k=1}^\infty q([|f_k|])^2$ is not norm convergent. Indeed, take $m_0$ such that $r_{m_0}\in V$, then we can repeat the previous argument.  
For any $N\in\mathbb N$ we can find $n_0\in\mathbb N$ such that $m(N+n_0)=m_0$. Then, for any $M>N+n_0$, we have $g_M-g_N\geq\frac{1}{2^{m_0}}h_{m_0}^{k(N+n_0)}$, hence 
$$
\|[g_M]-[g_N]\|\geq\frac{1}{2^{m_0}}\|[h_{m_0}^{k(N+n_0)}]\|=\frac{1}{2^{m_0}}.
$$

Let $a\in B(X)$ such that $a\neq 0$ and $fa\in l_2(A)$. Set $b=aa^*$, then $b\neq 0$, $b\geq 0$ and $fb\in l_2(A)$, so we may replace $a$ by $b$. If we show that $b=0$ then $a=0$ as well.  

Recall that every equivalence class in $B(X)$ has a lower semicontinuous representative \cite{Dixmier},
so there exists a lower semicontinuous $\beta\in Bor(X)$ such that $[\beta]=b$. Let 
$$
f(t)=\left\lbrace\begin{array}{cl}t&\mbox{if\ }t\geq 0;\\0&\mbox{if\ }t<0.\end{array}\right.
$$ 
Then $[f\circ\beta]=f([\beta])=f(b)=b$. As $f$ is non-decreasing, it is easy to see that $\alpha=f\circ\beta\in Bor(X)$ is also lower semicontinuous. We also have $\alpha(x)\geq 0$ for any $x\in X$.
As $b\neq 0$, there exists $\varepsilon>0$ and a non-empty open set $U\subset X$
such that $\alpha(x)>\varepsilon$ for any $x\in U$. Let $V\subset U$ be a non-empty closed subset. If the sequence $\{[g_n]b\}_{n\in\mathbb N}$ is norm convergent in $B(X)$ (i.e. if $fb\in l_2(A)$) then the sequence $\{q([g_n])\}_{n\in\mathbb N}$ is norm convergent in $B(V)$ too  --- a contradiction.

\end{proof} 

\begin{cor}
There exists a monotone complete commutative $C^*$-algebra $A$, a Hilbert $C^*$-module $N$ over $A$, a thick submodule $M\subset N$, and $n\in N$ such that $\|n\|_N>\|\widehat n|_M\|_{M'}$. 

\end{cor}
\begin{proof}
Let $A$ be as above, $M=l_2(A)$. Take $f\in l_2(A)'$ as in Lemma \ref{Borel}, i.e. such that $J_f=0$.
Then apply Theorem \ref{neq}.

\end{proof}

There is an interesting consequence of this result. We show that there exist functionals in $N'$ that equal zero on $M$.

Let $M\subset N$, and suppose that both $M$ and $N$ are Hilbert $C^*$-modules with a Hilbert dual. 
Denote the inclusion $M\to N$ by $\alpha$, and define the map $P:N'\to M'$ by $P(g)(m)=g(\alpha(m))$, $g\in N'$, $m\in M$. As the dual modules are Hilbert $C^*$-modules, this can be written as 
\begin{equation}\label{1}
\langle \widehat m,P(g)\rangle_{M'}=\langle \widehat{\alpha(m)},g\rangle_{N'}.
\end{equation}

As the dual modules are Hilbert $C^*$-modules, they are self-dual, i.e. equal to their second duals. Therefore, all bounded operators are adjointable. Let $I=P^*:M'\to N'$. Then $\langle f,P(g)\rangle_{M'}=\langle I(f),g\rangle_{N'}$, $f\in M'$. Taking $f=\widehat m$, $m\in M$, we get 
\begin{equation}\label{2}
\langle \widehat m,P(g)\rangle_{M'}=\langle I(\widehat m),g\rangle_{N'}.
\end{equation}  

Comparing (\ref{1}) and (\ref{2}), we see that $I(\widehat m)=\widehat{\alpha(m)}$ for any $m\in M$.

\begin{thm}\label{Thm}
Let $M\subset N$ be a thick submodule, let both $M$ and $N$ are Hilbert $C^*$-modules with a Hilbert dual. Then the map $I:M'\to N'$ is an isometric inclusion. If there exists $n\in N$ such that $\|n\|_N>\|\widehat n|_M\|_{M'}$ then $N'=I(M')\oplus K$ for some non-zero $K$. 

\end{thm}
\begin{proof}

Let us calculate the composition $PI:M'\to M'$.
For $m,m'\in M$, we have 
\begin{eqnarray*}
\langle PI(\widehat m),\widehat m'\rangle_{M'}&=&\langle I(\widehat m),I(\widehat m')\rangle_{N'}\\
&=&\langle \widehat{\alpha(m)},\widehat{\alpha(m')}\rangle_{N'}\\
&=&\langle\alpha(m),\alpha(m')\rangle_N\\
&=&\langle m,m'\rangle_M\\
&=&\langle \widehat m,\widehat m'\rangle_{M'},
\end{eqnarray*}
hence, $\langle PI(\widehat m)-\widehat m,\widehat m'\rangle_{M'}=0$ for any $m'\in M$, thus, as $M\subset M'$ is thick, we conclude that $PI(\widehat m)=\widehat m$ for any $m\in M$. 

Note that $PI$ is selfadjoint, and $(PI)|_M=\id_M$. This suffices to show that $PI=\id_{M'}$: 
$$
\langle PI(f),\widehat m\rangle_{M'}=\langle f,PI(\widehat m)\rangle_{M'}=\langle f,\widehat m\rangle_{M'},\quad f\in M',
$$
hence, by thickness of $M$ in $M'$, we conclude that $PI(f)=f$ for any $f\in M'$. 

Then $IP:N'\to N'$ is a projection with range $M'$, and both $P$ and $I$ are contractions. It follows from (\ref{1}) that 
$$
\langle \widehat m,P(\widehat n)\rangle_{M'}=\langle\widehat{\alpha(m)},\widehat n\rangle_{N'}=\langle \alpha(m),n\rangle_N=\widehat n|_M(m)
$$
for any $m\in M$, $n\in N$, hence $P(\widehat n)=\widehat n|_M$ for any $n\in N$.

Let $n\in N$ satisfy $\|n\|_N>\|\widehat n|_M\|_{M'}$. Then $\|P(\widehat n)\|_{M'}<\|n\|_N$, hence $\|IP(\widehat n)\|<\|n\|_N$. If $IP=\id_{N'}$ then this cannot be true, hence the projection $IP$ is not the identity map.
Set $K=(\id_{N'}-IP)(N')$.

\end{proof}

\begin{cor}
Let $M\subset N$ be a thick submodule, and let both $M$ and $N$ be Hilbert $C^*$-modules with a Hilbert dual. If there exists $n\in N$ such that $\|n\|_N > \|\widehat n|_M\|_{M'}$ then there
exists $g\in N'$ such that $g|_M=0$.

\end{cor}
\begin{proof} 
Let $K$ be as in Theorem \ref{Thm}, and let $g\in K$, $g\neq 0$. Then $\langle g,f\rangle'=0$ 
for any $f\in M'$, hence $g|_M=0$.

\end{proof}


\end{document}